\documentclass[12pt,reqno]{article}

\usepackage[usenames]{color}
\usepackage{amssymb}
\usepackage{amsmath}
\usepackage{amsthm}
\usepackage{amsfonts}
\usepackage{amscd}
\usepackage{graphicx}

\usepackage[colorlinks=true,
linkcolor=webgreen,
filecolor=webbrown,
citecolor=webgreen]{hyperref}

\definecolor{webgreen}{rgb}{0,.5,0}
\definecolor{webbrown}{rgb}{.6,0,0}

\usepackage{color}
\usepackage{fullpage}
\usepackage{float}

\usepackage{graphics}
\usepackage{latexsym}

\DeclareMathOperator{\arctanh}{arctanh}

\newcommand{\braces}{\genfrac{\lbrace}{\rbrace}{0pt}{}}

\begin{document}

\theoremstyle{plain}
\newtheorem{theorem}{Theorem}
\newtheorem{corollary}[theorem]{Corollary}
\newtheorem{lemma}{Lemma}
\newtheorem{example}{Example}
\newtheorem{remark}{Remark}

\begin{center}
\vskip 0.5cm{\LARGE\bf 
Binomial Fibonacci sums from Chebyshev polynomials \\
}
\vskip 1cm
{\large 
Kunle Adegoke \\
Department of Physics and Engineering Physics, \\ Obafemi Awolowo University, 220005 Ile-Ife, Nigeria \\
\href{mailto:adegoke00@gmail.com}{\tt adegoke00@gmail.com}

\vskip 0.15 in

Robert Frontczak \\
Independent Researcher \\ Reutlingen, Germany \\
\href{mailto:robert.frontczak@web.de}{\tt robert.frontczak@web.de}

\vskip 0.15 in

Taras Goy  \\
Faculty of Mathematics and Computer Science\\
Vasyl Stefanyk Precarpathian National University, Ivano-Frankivsk, Ukrai\-ne\\
\href{mailto:taras.goy@pnu.edu.ua}{\tt taras.goy@pnu.edu.ua}
}

\end{center}

\vskip 0.15cm

\begin{abstract} 
We explore new types of binomial sums with Fibonacci and Lucas numbers. The binomial coefficients under consideration are $\frac{n}{n+k}\binom{n+k}{n-k}$ and $\frac{k}{n+k}\binom{n+k}{n-k}$.
The identities are derived by relating the underlying sums to Chebyshev polynomials. Finally, some combinatorial sums are studied and a connection to a recent paper by Chu and Guo from 2022 is derived.  

\vskip 0.15cm 

\noindent 2020 {\it Mathematics Subject Classification}: Primary 11B39; Secondary 11B37.
\vskip 0.1cm 
\noindent \emph{Keywords:} Fibonacci (Lucas) number, Chebyshev polynomial, binomial coefficient.
\end{abstract}

\vskip 0.25cm

\section{Preliminaries}

As usual, the Fibonacci numbers $F_n$ and the Lucas numbers $L_n$ are defined, for \text{$n\in\mathbb Z$}, 
through the recurrence relations $F_n = F_{n-1}+F_{n-2}$, $n\ge 2$, with initial values $F_0=0$, $F_1=1$ and 
$L_n = L_{n-1}+L_{n-2}$ with $L_0=2$, $L_1=1$. For negative subscripts we have $F_{-n}=(-1)^{n-1}F_n$ and $L_{-n}=(-1)^n L_n$.
They possess the explicit formulas (Binet forms)
\begin{equation}\label{binet}
F_n = \frac{\alpha^n - \beta^n }{\alpha - \beta },\quad L_n = \alpha^n + \beta^n,\quad n\in\mathbb Z.
\end{equation}

The sequences $(F_n)_{n\geq 0}$ and $(L_n)_{n\geq 0}$ are indexed in the On-Line Encyclopedia of Integer Sequences 
\cite{OEIS} as entries A000045 and A000032, respectively. For more information we refer to Koshy \cite{Koshy} and Vajda \cite{Vajda} 
who have written excellent books dealing with Fibonacci and Lucas numbers. 

For any integer $n\geq0$, the Chebyshev polynomials $\{T_n(x)\}_{n\geq0}$ of the first kind are defined 
by the second-order recurrence  relation \cite{Mason}
\begin{equation*}
T_0(x) = 1,\quad T_1(x) = x, \quad T_{n+1}(x) = 2x T_n(x) - T_{n-1}(x),
\end{equation*}
while the Chebyshev polynomials $\{U_n(x)\}_{n\geq0}$ of the second kind are defined by
\begin{equation*}
U_0(x) = 1,\quad U_1(x) = 2x, \quad U_{n+1}(x) = 2xU_n(x) - U_{n-1}(x).
\end{equation*}
The Chebyshev polynomials possess the representations    
\begin{gather}
T_n(x) = \sum_{k=0}^{\lfloor{n}/{2}\rfloor} {n\choose 2k}(x^2-1)^kx^{n-2k}, \label{ChebRep1} \\
U_n(x) = \sum_{k=0}^{\lfloor{n}/{2}\rfloor} {n+1\choose 2k+1}(x^2-1)^kx^{n-2k}. \label{ChebRep2} 
\end{gather}

Also, sequences $T_n(x)$ and $U_n(x)$ have the Binet-like  formulas
\begin{gather}
T_n (x) = \frac{1}{2} \left( (x + \sqrt {x^2 - 1} )^n + (x - \sqrt {x^2 - 1} )^n \right)\!, \label{eq.lgfu65w} \\
U_n (x) = \frac{1}{2\sqrt {x^2  - 1}} \left( (x + \sqrt {x^2 - 1} )^{n + 1} - (x - \sqrt {x^2 - 1} )^{n + 1} \right)\!.
\end{gather}

The Chebyshev polynomials of the first and second kind are connected by
\begin{equation*}
T_{n + 1}(x) = xT_n(x) - (1 - x^2)U_{n - 1}(x)
\end{equation*}
and
\begin{equation}\label{eq.w83lskz}
U_{n + 1}(x) = xU_n(x) + T_{n + 1}(x);
\end{equation}
from which we also get
\begin{equation*}
T_n(x)=\frac12\big(U_n(x) - U_{n - 2}(x)\big).
\end{equation*}

The properties of Chebyshev polynomials have been studied extensively in the literature. 
The reader can find in the recent papers \cite{Abd,Fan,Frontczak,Gao,Kilic2,Kim,Li-Wenpeng,Li,Zhang}
additional information about them, especially about their products, convolutions, power sums as well as their connections to Fibonacci numbers and polynomials. 

There exists a countless number of binomial sums identities involving Fibonacci and Lucas numbers. For some new articles in this field we refer to the papers \cite{Adegoke1,Adegoke2,Adegoke3,Bai}.
In this paper, we will deal with sums of the following forms
\begin{equation*}
\sum_{k=0}^n \frac {n}{n + k} \binom{n + k}{n - k} x_k \qquad\mbox{and}\qquad \sum_{k=0}^n \frac {k}{n + k} \binom{n + k}{n - k} x_k,
\end{equation*}
where $x_k$ will be some weighted Fibonacci (Lucas) entries. Among the huge amount of Fibonacci (Lucas) sums
existing in the literature we couldn't find references treating these forms, although the binomial coefficients remind us of coefficients of Girard--Waring type \cite{Gould}, Jennings \cite{Jennings}, and also Kilic and Ioanescu \cite{Kilic1}. 
It may be also of interest that we can write the sums under consideration equivalently involving three binomial coefficients. For instance, the first type of sums equals (the second is similar)
\begin{equation*}
\sum_{k=0}^n  \frac{\binom {n+k-1} {k} \binom {n}{k}}{\binom {2k}{k}} x_k.
\end{equation*}

\section{Binomial Fibonacci and Lucas sums from identities involving $T_n(x)$}

We start by deriving two identities involving $T_n(x)$, which we prove for the readers' convenience. 
\begin{theorem}\label{thm1}
For all $x\in\mathbb{C}$ and non-negative integers $n$ and $m$ we have the following identities:
\begin{gather}\label{main_id1}
\sum_{k=0}^n (-2)^k\frac n{n + k}\binom{n + k}{n - k} (1\mp x)^k = (\pm 1)^n T_n(x),\\
\label{main_id2}
\sum_{k=0}^n (-2)^k \frac n{n + k}\binom{n + k}{n - k} \big(1\mp T_m(x)\big)^k = (\pm 1)^n T_{n m}(x).
\end{gather}
\end{theorem}
\begin{proof} It suffices to prove \eqref{main_id1}. Let $P_n^{(a,b)}(x)$, $a,b>-1$, be the Jacobi polynomial. Then $P_n^{(a,b)}(x)$
has the truncated series representation
\begin{equation*}
P_n^{(a,b)}(x) = \frac{\Gamma(a+n+1)}{ n!\Gamma(a+b+n+1)} \sum_{k=0}^n \binom {n}{k} \frac{\Gamma(a+b+n+k+1)}{2^k \Gamma(a+k+1)} (x-1)^k.
\end{equation*}
Now, we use the fact that
$$T_n(x) = \frac{4^n (n!)^2}{(2n)!} P_n^{(-1/2,-1/2)}(x).
$$
The statement follows after some steps of simplifications using
$\Gamma\big(n+\frac{1}{2} \big) = \frac{(2n)! \sqrt{\pi}}{4^n n!}$.

Identity \eqref{main_id2} follows immediately if $x\rightarrow T_{m}(x)$, because $T_n\big(T_m(x)\big)=T_{nm}(x)$.
\end{proof}

We note that the relation $T_n(2x^2-1)=T_{2n}(x)$ in conjunction with \eqref{main_id1} yields
\begin{equation}\label{main_id3}
\sum_{k=0}^n 4^k \frac n{n + k}\binom{n + k}{n - k} (x^2-1)^k = T_{2n}(x)
\end{equation}
and
\begin{equation}\label{main_id4}
\sum_{k=0}^n (-4)^k \frac n{n + k}\binom{n + k}{n - k} x^{2k} = (-1)^n T_{2n}(x).
\end{equation}
\begin{example}
Setting $x=0$ and $x=-1$ in \eqref{main_id1} we get
\begin{gather*}
\sum_{k=0}^n (-2)^k \frac n{n + k}\binom{n + k}{n - k} = 
\begin{cases}
0, & \text{\rm $n$ odd;} \\ 
(-1)^{n/2}, & \text{\rm $n$  even,} 
\end{cases}\\
\sum_{k=0}^n (-4)^k \frac n{n + k}\binom{n + k}{n - k} = (-1)^{n}.
\end{gather*}
\end{example}
\begin{example}\label{example2}
Directly from \eqref{main_id1} we have the following interesting sums
\begin{gather*}
\sum_{k=0}^n (-2)^k \frac {n}{n + k} \binom{n + k}{n - k} L_k = T_n(\alpha) + T_n(\beta),\\
\sum_{k=0}^n (-2)^{k} \frac {n}{n + k} \binom{n + k}{n - k} F_k =  - \frac{T_n(\alpha)-T_n(\beta)}{\sqrt{5}}.
\end{gather*}
It follows from these 
formulas that $T_n(\alpha)+T_n(\beta)$ and $\frac{T_n(\alpha)-T_n(\beta)}{\sqrt5}$ are integers. 
The sequence $\big\{\frac{T_n(\alpha)-T_n(\beta)}{\sqrt5}\big\}_{n\geq 0}=\big\{0, 1, 2, 5, 16, 45, 130, 377, 1088, 3145, \ldots\big\}$ 
is sequence A138573 in the OEIS \cite{OEIS}. But more is true as we show in the next theorem.
\end{example}
\begin{theorem}
For any integer $s$, the sequences $\{T_n(\alpha^s)+T_n(\beta^s)\}_{n\geq 0}$ and 
$\big\{\frac{T_n(\alpha^s)-T_n(\beta^s)}{\sqrt5}\big\}_{n\geq 0}$ are integers. Moreover, the sequences
$\{U_n(\alpha^s)+U_n(\beta^s)\}_{n\geq 0}$ and $\big\{\frac{U_n(\alpha^s)-U_n(\beta^s)}{\sqrt5}\big\}_{n\geq 0}$
are also integer sequences for each $s$.
\end{theorem}
\begin{proof}
For complex variable $x$, from \eqref{ChebRep1} we have
\begin{gather*}
T_n(\alpha^sx) = \sum_{k = 0}^{\lfloor{n/2}\rfloor} \sum_{j = 0}^{k}(-1)^{j} \binom{n}{2k} \binom{k}{j} (\alpha^sx)^{n-2j},\\
T_n(\beta^sx) = \sum_{k = 0}^{\lfloor{n/2}\rfloor} \sum_{j = 0}^{k}(-1)^{j} \binom{n}{2k}\binom{k}{j} (\beta^sx)^{n-2j}.
\end{gather*}

In particular, 
\begin{gather*}
T_n(\alpha^s) =
\begin{cases}
\sum\limits_{k = 0}^{\lfloor{n/2}\rfloor} \binom{n}{2k} L^k_s \alpha^{s(n-k)}, & \text{\rm $s$ odd;} \\ 
\sum\limits_{k = 0}^{\lfloor{n/2}\rfloor} \binom{n}{2k} (\sqrt5 F_s)^k \alpha^{s(n-k)}, & \text{\rm $s$  even,} 
\end{cases}\\
T_n(\beta^s) = 
\begin{cases}
\sum\limits_{k = 0}^{\lfloor{n/2}\rfloor} \binom{n}{2k} L^k_s \beta^{s(n-k)}, & \text{\rm $s$ odd;} \\ 
\sum\limits_{k = 0}^{\lfloor{n/2}\rfloor} \binom{n}{2k}(-\sqrt5 F_s)^k \beta^{s(n-k)}, & \text{\rm $s$  even.} 
\end{cases}
\end{gather*}

Thus, 
\begin{gather*}
T_n(\alpha^s)+T_n(\beta^s)=\sum_{k = 0}^{\lfloor{n/2}\rfloor} \binom{n}{2k}L^k_s L_{s(n-k)},\quad s\text{ odd},\\
\frac{T_n(\alpha^s)-T_n(\beta^s)}{\sqrt5}=\sum_{k = 0}^{\lfloor{n/2}\rfloor} \binom{n}{2k}L^k_s F_{s(n-k)},\quad s\text{ odd},
\end{gather*}
and
\begin{gather*}
T_n(\alpha^s) + T_n(\beta^s) = \sum_{k = 0}^{\lfloor{n/4}\rfloor} 5^k F^{2k}_s\left(\binom{n}{4k}L_{s(n-2k)} 
+ 5 \binom{n}{4k+2} F_s F_{s(n-2k-1)}\right)\!,\quad s\text{ even},\\
\frac{T_n(\alpha^s) - T_n(\beta^s)}{\sqrt5} = \sum_{k = 0}^{\lfloor{n/4}\rfloor} 5^k F^{2k}_s\left(\binom{n}{4k}F_{s(n-2k)} 
+ \binom{n}{4k+2} F_s L_{s(n-2k-1)}\right)\!,\quad s\text{ even}.
\end{gather*}

For Chebyshev polynomials of the second kind similar formulas hold true. From  \eqref{ChebRep2} using Binet's formulas \eqref{binet} we have
\begin{gather*}
U_n(\alpha^s)+U_n(\beta^s)=\sum_{k = 0}^{\lfloor{n/2}\rfloor} \binom{n+1}{2k+1}L^k_s L_{s(n-k)},\quad s\text{ odd},\\
\frac{U_n(\alpha^s)-U_n(\beta^s)}{\sqrt5}=\sum_{k = 0}^{\lfloor{n/2}\rfloor} \binom{n+1}{2k+1}L^k_s F_{s(n-k)},\quad s\text{ odd},
\end{gather*}
and
\begin{gather*}
U_n(\alpha^s) + U_n(\beta^s)=\sum_{k = 0}^{\lfloor{n/4}\rfloor}  
 5^k F^{2k}_s \left(\binom{n+1}{4k+1}L_{s(n-2k)} + 5 \binom{n+1}{4k+3} F_s F_{s(n-2k-1)}\right)\!,\quad s\text{ even},\\
\frac{U_n(\alpha^s) - U_n(\beta^s)}{\sqrt5} = \sum_{k = 0}^{\lfloor{n/4}\rfloor}  
5^k F^{2k}_s\left(\binom{n+1}{4k+1}F_{s(n-2k)} + \binom{n+1}{4k+3} F_s L_{s(n-2k-1)}\right)\!,\quad s\text{ even}.
\end{gather*}
\end{proof}
\begin{example} Here we proceed with additional interesting sums of the same kind as in Example~\ref{example2}:
\begin{gather*}
\sum_{k=0}^n (-2)^k \frac {n}{n + k} \binom{n + k}{n-k} L_{2k} = (-1)^n \big(T_n(\alpha)+T_n(\beta)\big),\\
\sqrt5 \sum_{k=0}^n  (-2)^k \frac {n}{n + k} \binom{n + k}{n-k} F_{2k} = (-1)^n\big(T_n(\alpha)-T_n(\beta)\big),\\
\sum_{k=0}^n (-2)^k \frac {n}{n + k}\binom{n + k}{n-k} L_{3k} = (-1)^n \big(T_n(2\alpha)+T_n(2\beta)\big),\\
\sum_{k=0}^n (-2)^k \frac {n}{n + k}\binom{n + k}{n-k} F_{3k} = \frac{(-1)^n}{\sqrt5} \big(T_n(2\alpha)-T_n(2	\beta)\big),\\
\sum_{k=0}^n  \frac {n}{n + k} \binom{n + k}{n-k} L_k  = T_n\Big(\frac{\sqrt5\alpha}2\Big)+T_n\Big(\frac{\sqrt5\beta}2\Big),\\
\sum_{k=0}^n  \frac {n}{n + k} \binom{n + k}{n-k} F_k  =\frac{1}{\sqrt5}\Big( T_n\Big(\frac{\sqrt5\alpha}2\Big)-T_n\Big(\frac{\sqrt5\beta}2\Big)\Big),\\
\sum_{k=0}^n  4^k \frac {n}{n + k} \binom{n + k}{n-k} L_{k} = T_n(2\alpha)+T_n(2\beta),\\
\sum_{k=0}^n 4^k \frac {n}{n + k}\binom{n + k}{n-k} F_{k} = \frac{1}{\sqrt5}\big(T_n(2\alpha)-T_n(2\beta)\big).
\end{gather*}
\end{example}

By substituting $x=L_p/2$ and $x=\sqrt 5F_p/2$ in~\eqref{eq.lgfu65w}, in turn, we obtain the results stated in Lemma~\ref{lem.v25emwx}.
\begin{lemma}\label{lem.v25emwx}
If $p$ is an integer, then
\begin{gather}
T_n\Big(\frac{L_p}2\Big) =\frac12L_{pn},\quad  \text{$p$ \rm even,}\label{eq.xgyt5jf} \\
T_n \Big( {\frac{\sqrt 5F_p}{2}} \Big) = \begin{cases}
\frac12 L_{pn},&\text{\rm $p$ odd, $n$ even;}  \\ 
\frac{\sqrt 5}{2} F_{pn},&\text{\rm $p$ odd, $n$ odd.} \label{eq.r1jl3d5}
 \end{cases}
 \end{gather}
In particular,
\begin{gather}
T_n\Big(\frac32\Big)=\frac12L_{2n},\nonumber\\
T_n \Big( {\frac{\sqrt 5 }{2}} \Big) =  \begin{cases}
 \frac12L_{n},&\text{\rm $n$ even;}  \\ 
 \frac{\sqrt 5}2F_n , &\text{\rm  $n$ odd;}  \\ 
 \end{cases}\label{eq.a8aqw1m}\\
T_n \big(\sqrt 5\big) =  \begin{cases}
\frac12 L_{3n},&\text{\rm  $n$ even;}  \\ 
 \frac{\sqrt 5}2F_{3n}, & \text{\rm  $n$ odd.} 
 \end{cases}\label{eq.l6oienp}
\end{gather}
\end{lemma}
\begin{theorem}\label{thm_xxx}
If $n$ is a positive integer and $p$ is an integer, then
\begin{gather*}
\sum_{k = 0}^n (-1)^{(p-1)(n-k)} \frac{n}{n + k} \binom{n + k}{n-k}  L_p^{2k} = \frac{L_{2pn}}{2},\\
\sum_{k = 0}^n (-1)^{p(n-k)} \frac{n}{n + k} \binom{n + k}{n-k} 5^k F_p^{2k} = \frac{L_{2pn}}{2}.
\end{gather*}
\end{theorem}
\begin{proof}
Set $x=L_{2p}/2$ in~\eqref{main_id1}, use~\eqref{eq.xgyt5jf} and the fact that
\begin{equation}\label{eq.r1j2cpo}
L_{2p} - 2 =  \begin{cases}
 5F_p^2,&\text{$p$ even}; \\ 
 L_p^2, &\text{$p$ odd};\\ 
 \end{cases}\qquad 
L_{2p} + 2 =  \begin{cases}
 5F_p^2,& \text{$p$ odd};\\ 
 L_p^2,& \text{$p$ even};\\ 
 \end{cases}
\end{equation}
to get
\begin{gather*}
\sum_{k = 0}^n \frac{n}{n + k} \binom{n + k}{n - k} ( \pm 1)^k L_p^{2k} = \frac{{( \pm 1)^n }}{2}L_{2pn}\quad\braces{\mbox{$p$ odd}}{\mbox{$p$ even}},\\
\sum_{k = 0}^n \frac{n}{n + k} \binom{n + k}{n - k} ( \mp 5)^k F_p^{2k} = \frac{{( \mp 1)^n }}{2}L_{2pn}\quad\braces{\mbox{$p$ odd}}{\mbox{$p$ even}},
\end{gather*}
from which the stated identities follow. 
\end{proof}

Theorem \ref{thm_xxx} can be generalized in the following way.
\begin{theorem}
If $n$ is a positive integer and $p$ is an integer, then we have
\begin{gather*}
\sum_{k = 0}^n (-1)^{n-k} \frac{n}{n + k} \binom{n + k}{n - k}(2 \pm \sqrt{5}F_{pm})^k = \frac{\pm \sqrt{5}}{2}  F_{p m n},
\quad \text{\rm $p$, $m$, $n$ odd}, \\
\sum_{k = 0}^n (-1)^{n-k} \frac{n}{n + k} \binom{n + k}{n - k}(2 \pm L_{pm})^k = \frac{( \pm 1)^n}{2} L_{p m n},
\quad \text{\rm $p$ odd, $m$ even}, \\
\sum_{k = 0}^n (-1)^{n-k} \frac{n}{n + k} \binom{n + k}{n - k} (2 \pm L_{pm})^k = \frac{( \pm 1)^n }{2} L_{p m n},\quad \text{\rm $p$ even.} 
\end{gather*}
\end{theorem}
\begin{proof}
Combine \eqref{main_id2} with \eqref{eq.xgyt5jf} and \eqref{eq.r1jl3d5}.
\end{proof}

Some particular cases of Theorems 3 and 4 stated in the next Example.
\begin{example}\label{ex.lhwdd75} We have
\begin{gather*}
\sum_{k=0}^n \frac{n}{n + k}\binom{n + k}{n - k} \alpha^{-3k} = \sum_{k=0}^n (-1)^{k+1} \frac{n}{n + k} \binom{n + k}{n - k} \alpha^{3k} 
= \frac{\sqrt{5}F_{n}}{2}, \quad \text{\rm $n$ odd},\\
\sum_{k=0}^n \frac{n}{n + k}\binom{n + k}{n - k} 4^k \alpha^{-k} = \sum_{k=0}^n (-1)^{k+1} \frac{n}{n + k} \binom{n + k}{n - k} 4^k \alpha^{k} 
= \frac{\sqrt{5}F_{3n}}{2}, \quad \text{\rm $n$ odd},
\end{gather*}
\begin{gather*}
\sum_{k=0}^n \frac{n}{n + k}\binom{n + k}{n - k} = \frac{L_{2n}}{2},\qquad
\sum_{k=0}^n (-1)^{n-k} \frac{n}{n + k}\binom{n + k}{n - k} 5^k = \frac{L_{2n}}{2},\\
\sum_{k=0}^n \frac{n}{n + k}\binom{n + k}{n - k} 5^k = \frac{L_{4n}}{2},\qquad
\sum_{k=0}^n (-1)^{n-k} \frac n{n + k}\binom{n + k}{n - k} 9^k = \frac{L_{4n}}{2}.
\end{gather*}
\end{example}
\begin{lemma}
If $p$ is an integer, then
\begin{gather}
T_n \Big( {\frac{{ \sqrt 5F_p }}{L_p }} \Big) = \cosh \left( {n\arctanh\Big( {\frac{2}{ \sqrt 5F_p }}\Big)} \right),\quad\text{\rm $p$ odd},\label{eq.qefdypb}\\
T_n \Big( {\frac{L_p}{ \sqrt 5F_p }} \Big) = \cosh \left( {n\arctanh\Big( {\frac{2}{{L_p }}} \Big)} \right),\quad\text{\rm $p$ even, $p\ne 0$}.\label{eq.y8mmtsi}
\end{gather}
\end{lemma}
\begin{proof}
Setting $x=\sqrt 5F_p/L_p$ in~\eqref{eq.lgfu65w} and making use of $5F_p^2 - 4(-1)^{p + 1}=L_p^2$ with $p$ odd produces
\begin{equation}\label{eq.h51c7s8}
T_n \Big( {\frac{ \sqrt 5F_p }{L_p}} \Big) = \frac{{(\sqrt 5F_p   - 2)^n  + (\sqrt 5F_p + 2)^n }}{2L_p^n }
\end{equation}
from which~\eqref{eq.qefdypb} follows upon using the identity
\begin{equation*}
(x - y)^n  + (x + y)^n  = 2\big( {\sqrt {x^2  - y^2 } } \big)^n \cosh \left( {n\arctanh\left( {\frac{y}{x}} \right)} \right).
\end{equation*}
\end{proof}

For low values of $p$, it is easier to use \eqref{eq.h51c7s8} directly for evaluation. 
Thus, at $p=1$ we recover \eqref{eq.l6oienp} while $p=3$ gives \eqref{eq.a8aqw1m}.
On account  of \eqref{eq.r1j2cpo}, \eqref{eq.y8mmtsi} also implies
\begin{equation}\label{eq.wry7jin}
T_n \Big ( {\frac{{L_{2p} }}{\sqrt 5F_{2p} }} \Big) = \frac{{(5F_p^2 )^n + (L_p^2 )^n }}{{2(\sqrt 5F_{2p}  )^n }},
\end{equation}
for every non-zero integer $p$. We also note that
\begin{equation*}
\sum_{k = 0}^{2n} \Big (\frac{4}{5}\Big )^k \binom{2n + k}{2n-k} \frac {4^k - (-1)^k L_{2p}^{2k}}{(2n + k){F_{2p}^{2k}}} = 0,\quad p\ne0,
\end{equation*}
and
\begin{equation*}
\sum_{k = 0}^{2n} \Big (\frac{4}{5}\Big )^k  \binom{2n + k}{2n-k} \frac {4^k + (-1)^k }{2n +k}\Big(\frac{L_{2p}}{F_{2p}}\Big)^{2k}  
= \frac{625^k F_p^{8n}+L_p^{8n}}{2n \,25^n F_{2p}^{4n}}, \quad p\ne0.
\end{equation*}
\begin{theorem}
If $p$ is a non-zero integer and $n$ is a positive integer, then 
\begin{equation*}
\sum_{k = 0}^n \Big ( \frac{16}{5}\Big )^k \frac {n}{n + k} \binom{n + k}{n - k} F_{2p}^{2(n-k)} 
= \frac{25^n F_p^{4n} + L_p^{4n}}{2\cdot5^n}.
\end{equation*}
\begin{equation*}
\sum_{k = 0}^n (-1)^{n-k} \Big ( \frac{4}{5}\Big )^k \frac {n}{n + k} \binom{n + k}{n - k}\Big ( \frac{F_{2p}}{L_{2p}}\Big )^{2(n-k)} 
= \frac{25^n F_p^{4n} + L_p^{4n}}{2\cdot 5^n L_{2p}^{2n}}.
\end{equation*}
In particular,
\begin{gather*}
\sum_{k = 0}^n \Big ( \frac{16}{5}\Big )^k \frac {n}{n + k} \binom{n + k}{n - k} = \frac{25^n + 1}{2\cdot 5^n},\\
\sum_{k = 0}^n (-1)^{n-k} \Big ( \frac{36}{5}\Big )^k \frac {n}{n + k} \binom{n + k}{n - k} = \frac{25^n + 1}{2\cdot 5^n}.
\end{gather*}
\end{theorem}
\begin{proof}
Combine \eqref{eq.wry7jin} with \eqref{main_id3} and \eqref{main_id4}, respectively, while setting $x=L_{2p}/(\sqrt{5}F_{2p})$.
\end{proof}
\begin{remark}
We note the following  relations:
\begin{gather*}
\sum_{k = 0}^n \Big ( \frac{5}{16}\Big )^{n-k} \frac {n}{n + k} \binom{n + k}{n - k} F_{2p}^{2(n-k)}
=  \sum_{k=0}^n \binom {2n}{2k} \frac{L_{2p}^{2(n-k)}}{4^{2n-k}},\\
\sum_{k = 0}^n (-1)^{n-k} \Big ( \frac{5}{4}\Big )^{n-k} \frac {n}{n + k} \binom{n + k}{n - k} \Big (\frac{F_{2p}}{L_{2p}}\Big )^{2(n-k)}
= \sum_{k=0}^n \binom {2n}{2k} \frac{4^{k-n}}{L^{2k}_{2p}}.
\end{gather*}
\end{remark}
\begin{theorem}\label{thm.fepmhxp}
If $p$ is an odd integer and $n$ is a positive integer, then
\begin{gather*}
\sum_{k = 0}^n  (- 4)^k \frac {n}{n + k}\binom{n + k}{n - k} \frac{L_{pk + t}}{L_p^{k}} = \cosh\!\left( {n\arctanh\left( {\frac{2}{\sqrt 5 F_p}} \right)} \right) \cdot 
 \begin{cases}
  L_t,&\text{\rm $n$ even};  \\ 
  - \sqrt 5 F_t ,&\text{\rm $n$ odd};   
 \end{cases}\\ 
\sum_{k = 0}^n (- 4)^k \frac {n}{n + k}\binom{n + k}{n - k} \frac{F_{pk + t}}{L_p^{k}} = \cosh\! \left( {n\arctanh\left( {\frac{2}{\sqrt 5 F_p  }} \right)} \right) \cdot  \begin{cases}
 F_t,&\text{\rm $n$ even};  \\ 
  - L_t /\sqrt 5,&\text{\rm $n$ odd}. \\ 
 \end{cases} 
\end{gather*}
\end{theorem}
\begin{proof} Set $x=F_p\sqrt 5/L_p$ in~\eqref{main_id1}, taking the upper signs, and then  use \eqref{eq.qefdypb}.
\end{proof}
\begin{corollary}
If $p$ is an odd integer and $n$ is a positive integer, then
\begin{gather*}
\sum_{k = 0}^n (- 4)^k \frac {n}{n + k}\binom{n + k}{n - k} \frac{L_{pk}}{L_p^{k}} = 0,\quad\text{\rm $n$ odd}, \\
\sum_{k = 0}^n (- 4)^k \frac {n}{n + k}\binom{n + k}{n - k} \frac{F_{pk}}{L_p^{k}} = 0,\quad\text{\rm $n$ even}.
\end{gather*}
\end{corollary}
\begin{corollary}
If $n$ is a positive integer, then
\begin{gather*}
\sum_{k = 0}^n (- 4)^k \frac {n}{n + k}\binom{n + k}{n - k} L_{k + t} = \begin{cases}
 \frac12 L_tL_{3n},&\text{\rm $n$ even};  \\ 
  - \frac52F_tF_{3n},&\text{\rm $n$ odd};  \\ 
 \end{cases} \\
\sum_{k = 0}^n (-4)^k \frac {n}{n + k}\binom{n + k}{n - k}F_{k + t} = \begin{cases}
 \frac12 F_tL_{3n},&\text{\rm $n$ even};  \\ 
  - \frac12 L_tF_{3n},&\text{\rm $n$ odd}. 
 \end{cases} 
\end{gather*}
\end{corollary}

The corresponding identities to Theorem \ref{thm.fepmhxp} for $p$ even are stated next.
\begin{theorem} If $p$ is a non-zero even integer,  $n$ is a positive integer and $t$ is an integer, then
\begin{align*}
\sum_{k = 0}^{\lfloor{n/2}\rfloor}&  \Big(\frac{16}{5F_p^2}\Big)^k
\bigg( \binom{n+2k}{n-2k}\frac{L_{2pk+t}}{n+2k}-\frac{4}{F_p}\binom{n+(2k+1)}{n-(2k+1)}\frac{F_{p(2k+1)+t}}{n+2k+1}\bigg) \\
&=\frac{1}{n}\cosh\left( n\arctanh\Big( \frac{2}{L_p}\Big) \right)\cdot
	\begin{cases}
	L_t,&\text{\rm$n$ even};  \\ 
	- \sqrt5 F_t,&\text{\rm$n$ odd};
	\end{cases}\\
\sum_{k = 0}^{\lfloor{n/2}\rfloor}&   \Big(\frac{16}{5F^2_p}\Big)^k
\left( \binom{n+2k}{n-2k}\frac{F_{2pk+t}}{n+2k}-\frac{4}{5F_p}\binom{n+(2k+1)}{n-(2k+1)}\frac{L_{p(2k+1)+t}}{n+2k+1}\right) \\
&=\frac{1}{\sqrt5\, n}\cosh\left( n\arctanh\Big( \frac{2}{L_p}\Big) \right)\cdot
\begin{cases}
	\sqrt5 F_t,&\text{\rm$n$ even};  \\ 
	- L_t,&\text{\rm$n$ odd}.  
	\end{cases}
\end{align*}
\end{theorem}

The non-alternating versions of the identities in Theorem~\ref{thm.fepmhxp} will be derived from~\eqref{eq.qefdypb} and~\eqref{main_id1} by considering the bottom sign combinations. This is left to the interested reader.
\begin{example}\label{eq.e5fgbi7}
Noting that 
\begin{equation}\label{help_ch_id1}
\sum_{k=0}^n (-2)^k \frac n{n + k}\binom{n + k}{n - k} \big ( (1+x)^k \mp (1-x)^k \big )= \big((-1)^n \mp 1\big)T_n(x),
\end{equation}
with $x=\sqrt{5}/2$ and \eqref{eq.a8aqw1m}  
we get
\begin{equation}\label{eq.hpzpjoo}
\sum_{k=1}^n (-1)^{k-1} \frac n{n + k}\binom{n + k}{n - k} F_{3k} =\frac{1-(-1)^n}{2} F_n,
\end{equation}
as well as
\begin{equation}\label{eq.rsep1ey}
\sum_{k=0}^n (-1)^{k} \frac n{n + k}\binom{n + k}{n - k} L_{3k} = 
\frac{1+(-1)^n}{2}L_n. 
\end{equation}
\end{example}

As we will see identities \eqref{eq.hpzpjoo} and \eqref{eq.rsep1ey} are special cases of Theorem~\ref{thm.iq16bty} below.
\begin{example}
We derive the inverse relation of \eqref{eq.hpzpjoo} and \eqref{eq.rsep1ey}. Starting with \eqref{help_ch_id1}  
we set $x=-\sqrt{5}$, and make use of $T_n(-\sqrt{5})=(-1)^nT_n(-\sqrt{5})$ together with  \eqref{eq.l6oienp} to derive the following identities: 
\begin{gather*}
\sum_{k=1}^n (-4)^{k} \frac n{n + k}\binom{n + k}{n - k} F_{k} =\frac{(-1)^n-1}{2}F_{3n},\\
\sum_{k=0}^n (-4)^{k} \frac n{n + k}\binom{n + k}{n - k} L_{k} = 
\frac{1+(-1)^n}{2} L_{3n}.
\end{gather*}
\end{example}
\begin{example}
It is obvious that from \eqref{help_ch_id1}
more appealing relations can be derived. We give just four examples:
\begin{align*}
\sum_{k = 0}^{n} (-4)^k \frac{n}{n + k} \binom{n+k}{n-k}\big(L_{2k}\pm (-1)^k L_{k}\big) 
&= \big((-1)^n\pm 1\big) \big(T_n(\alpha^3)+T_n(\beta^3)\big) \\
&=\big((-1)^n\pm 1\big) \sum_{k = 0}^{\lfloor{n/2}\rfloor}\binom{n}{2k}4^kL_{3(n-k)},\\
\sum_{k = 0}^{n} (-4)^k \frac{n}{n + k} \binom{n+k}{n-k}\big(F_{2k}\pm(-1)^kF_{k}\big)
&= \frac{(-1)^n\pm1}{\sqrt{5}} \big(T_n(\alpha^3)-T_n(\beta^3)\big) \\
&= \sum_{k = 0}^{\lfloor{n/2}\rfloor}\binom{n}{2k}4^k F_{3(n-k)},\\
\sum_{k = 0}^{n} \Big(\!-\frac{4}{3}\Big)^k \frac{n}{n + k} \binom{n+k}{n-k} L_{2k} &= \big((-1)^n +1\big)T_n\Big(\frac{\sqrt5}3\Big),\\
\sum_{k = 0}^{n} \Big(\!-\frac{4}{3}\Big)^k \frac{n}{n + k} \binom{n+k}{n-k} F_{2k} &= \frac{(-1)^n -1}{\sqrt{5}}\,T_n\Big(\frac{\sqrt5}3\Big).
\end{align*}
\end{example}
\begin{theorem}\label{thm.iq16bty}
If $n$ is a positive integer and $t$ is any integer, then
\begin{gather*}
\sum_{k = 0}^n {( - 1)^{n-k} \frac{n}{{n + k}}\binom{n + k}{n - k}L_{3k + t} } = \begin{cases}
 \frac52 F_t F_n,&\text{\rm $n$ odd};  \\ 
 \frac12 L_t L_n,&\text{\rm $n$ even};  \\ 
 \end{cases} \\
\sum_{k = 0}^n {( - 1)^{n-k} \frac{n}{{n + k}}\binom{n + k}{n - k}F_{3k + t} } =  \begin{cases}
 \frac12 L_t F_n,&\text{\rm $n$ odd};  \\ 
 \frac12 F_t L_n,&\text{\rm $n$ even}.  \\ 
 \end{cases}
\end{gather*}
\end{theorem}
\begin{proof}
Set $x=\sqrt{\alpha^3}/2$ and $x=\sqrt{\beta^3}/2$, in turn, in~\eqref{main_id4} and use 
\begin{equation*}
T_{2n} \bigg( \frac{{\sqrt {\alpha ^3 } }}{2}\bigg) =  \begin{cases}
\frac{\sqrt5}2 F_n,&\text{\rm $n$  odd};  \\
\frac12 L_n,&\text{\rm $n$ even};  \\ 
\end{cases}\qquad 
T_{2n} \bigg( {\frac{{\sqrt {\beta ^3 } }}{2}} \bigg) =  \begin{cases}
-\frac{\sqrt5}2 F_n,&\text{\rm $n$ odd};  \\ 
\frac12 L_n,&\text{\rm $n$ even}.  \\ 
\end{cases}
\end{equation*}
to obtain
\begin{equation*}
2\sum_{k = 0}^n {( - 1)^{n-k} \frac{n}{{n + k}}\binom{n + k}{n - k}(\alpha ^{3k + t}  + \lambda \beta ^{3k + t} )}= \begin{cases}
 F_n \sqrt 5 (\alpha ^t  - \lambda \beta ^t ) ,& \text{$n$ odd};\\ 
 L_n (\alpha ^t - \lambda \beta ^t ),& \text{$n$ even};\\ 
 \end{cases} 
\end{equation*}
from which the stated results follow upon setting $\lambda=1$ and $\lambda=-1$.
\end{proof}

Note that Examples \ref{eq.e5fgbi7} is the special case ($t=0$) of Theorem~\ref{thm.iq16bty}.
\begin{theorem}
If $n$ is a positive integer and $p$ is a non-zero integer, then
\begin{equation*}
\sum_{k = 0}^n (- 1)^{n-k} \frac{n}{{n + k}} \binom{n + k}{n - k} \left( \frac{{2L_{2p} }}{\sqrt5F_{2p}} \right)^{2k} 
= \frac12\left(\Big( \frac{\sqrt5 F_p }{L_p } \Big)^{2n} + \Big( \frac{{L_p}}{\sqrt5 F_p } \Big)^{2n}\right)\!.
\end{equation*}
\end{theorem}
\begin{proof}
Set $ x=L_{2p}/(\sqrt 5 F_{2p})$ in~\eqref{main_id4} and then use~\eqref{eq.wry7jin}.
\end{proof}
%
%
%
\begin{theorem}
If $n$ is a non-negative integer and $p$ and $q$ are non-zero integers, then
\begin{equation*}
\sum_{k = 0}^n ( - 1)^{(p - q)(n-k)}  \frac{n}{{n + k}}\binom{n + k}{n - k}\left( {\frac{{F_{p + q} F_{p - q} }}{{F_p F_q }}} \right)^{2k}  = \frac{{F_p^{4n}  + F_q^{4n} }}{2F_q^{2n} F_p^{2n}}.
\end{equation*}
\end{theorem}
\begin{proof}
Set $\displaystyle x=\frac{{(-i)^{p-q+1}F_{p + q} F_{p - q} }}{2 F_q F_p }$ in~\eqref{main_id4} and use 
\begin{equation*}
T_n \left( {\frac{i^{p - q + 1}F_{p + q} F_{p - q}}{2 F_q F_p }} \right) = i^{n(p - q + 1)}\frac{{F_p^{2n}  + ( - 1)^{n(p - q + 1)} F_q^{2n} }}{2 F_q^n F_p^n }.
\end{equation*}
\end{proof}
%
\begin{theorem}
	If $n$ is a non-negative integer and $p$ is a non-zero integer, then
	\begin{gather*}
	\sum_{k = 0}^n ( - 1)^{p(n-k)} \frac n{{n + k}}\binom{n + k}{n - k}\left( {\frac{F_{3p}}{F_{2p}}} \right)^{2k}   = \frac{{L_p^{4n}  + 1}}{2L_p^{2n} },\\
	\sum_{k = 0}^n ( - 1)^{(p + 1)(n-k)} \frac{n}{n + k}\binom{n + k}{n - k}5^{n-k} \left( \frac{{L_{3p} }}{F_{2p} } \right)^{2k}   = \frac{5^{2n}F_p^{4n}   + 1}{2 F_p^{2n}}.
	\end{gather*}
\end{theorem}
\begin{proof}
	Set $\displaystyle
	x = \frac{(-i)^{p+1}{F_{3p} }}{{2F_{2p} }}$ and $\displaystyle x=\frac{(-i)^pL_{3p} }{2\sqrt 5F_{2p}}$,
	in turn, in~\eqref{main_id4} and use 
		\begin{gather*}
	T_n \left( \frac{(-i)^{p+1}{F_{3p} }}{2F_{2p}} \right) = (-i)^{n(p + 1)} \frac{L_p^{2n}  + ( - 1)^{n(p + 1)} }{2L_p^n},\\ 
	T_n \left(\frac{(-i)^pL_{3p} }{2\sqrt 5F_{2p}} \right) = (-i)^{np} \frac{{F_p^{2n} 5^n  + ( - 1)^{np} }}{2 \sqrt{5^n}F_p^n}.
	\end{gather*}
	\end{proof}
\begin{lemma}
If $p$ and $q$ are integers, then
\begin{gather}
T_{2n} \left( {\sqrt{ \frac{( - 1)^{q + 1}F_q^2 }{4F_p F_{p + q}}\alpha ^{2p + q} }} \right) = \frac{{( - 1)^{n} }}{2}\!\left((-1)^{nq}\frac{{F_{p + q}^n }}{{F_p^n }}\alpha ^{qn}  + \frac{{F_p^n }}{{F_{p + q}^n }}\beta ^{qn}\right)\!,\quad \text{$p\ne 0$, $p\ne -q$}, \label{eq.in8cl0x}\\
T_{2n} \left(\sqrt { \frac{( - 1)^{q + 1}F_q^2}{4F_p F_{p + q}}\beta ^{2p + q}} \right) = \frac{( - 1)^{n} }{2}\!\left((-1)^{nq}\frac{{F_{p + q}^n }}{{F_p^n }}\beta ^{qn}  + \frac{{F_p^n }}{{F_{p + q}^n }}\alpha ^{qn}\right)\!,\quad \text{$p\ne 0$, $p\ne -q$},\nonumber
\end{gather}
\begin{gather}
T_{2n} \left( \sqrt{\frac{5( - 1)^{q + 1}F_q^2 }{4L_p L_{p + q} }\alpha ^{2p + q} } \right) = \frac{( - 1)^{n}}{2}\!\left((-1)^{nq}\frac{{L_{p + q}^n }}{{L_p^n }}\alpha ^{qn}  + \frac{{L_p^n }}{{L_{p + q}^n }}\beta ^{qn}\right)\!, \nonumber\\
T_{2n} \left(\sqrt{ \frac{5(- 1)^{q + 1}F_q^2 }{4L_p L_{p + q} }\beta ^{2p + q} } \right) = \frac{( - 1)^{n}}{2}\!\left((-1)^{nq}\frac{{L_{p + q}^n }}{{L_p^n }}\beta ^{qn}  + \frac{{L_p^n }}{{L_{p + q}^n }}\alpha ^{qn}\right)\!.\nonumber 
\end{gather}
\end{lemma}
\begin{theorem}\label{thm.awwyoac}
If $n$ is a positive integer and $p$, $q$ and $t$ are any integers, then
\begin{gather}
\sum_{k = 0}^n { \frac{( - 1)^{(n-k)q}\,n}{{n + k}}\binom{n + k}{n - k}F_{p + q}^{n - k} F_p^{n - k} F_q^{2k} L_{k(2p + q) + t} }= \frac{1}{2}\left(F_{p + q}^{2n} L_{t+qn}  + F_p^{2n} L_{t-qn}\right)\!,\label{eq.sqpsz3a}\\
\sum_{k = 0}^n  \frac{(-1)^{(n-k)q}\,n}{n + k}\binom{n + k}{n - k} F_{p + q}^{n - k} F_p^{n - k} F_q^{2k} F_{k(2p + q) + t} = \frac{1}{2}\left( F_{p + q}^{2n} F_{t+qn}  + F_p^{2n} F_{t-qn}\right)\!,\label{eq.uqixdew}\\ 
\sum_{k = 0}^n { \frac{( - 1)^{(n-k)q}\, n}{{n + k}}\binom{n + k}{n - k} 5^k L_{p + q}^{n - k} L_p^{n - k} F_q^{2k} L_{k(2p + q) + t} }= \frac{1}{2}\left(L_{p + q}^{2n} L_{t+qn}  + L_p^{2n} L_{t-qn}\right)\!,\label{eq.tm4ir7g}\\
\sum_{k = 0}^n { \frac{( - 1)^{(n-k)q}\, n}{{n + k}}\binom{n + k}{n - k}5^k L_{p + q}^{n - k} L_p^{n - k} F_q^{2k} F_{k(2p + q) + t} }= \frac{1}{2}\left(L_{p + q}^{2n} F_{t+qn}  + L_p^{2n} F_{t-qn}\right)\!.\label{eq.h4xq8wt}
\end{gather}
\end{theorem}
\begin{proof}
In~\eqref{main_id4} set
$\displaystyle
x=\sqrt{ \frac{( - 1)^{q + 1}F_q^2 }{4F_p F_{p + q} }\alpha ^{2p + q}}
$
and use~\eqref{eq.in8cl0x} to obtain
\begin{equation*}
\sum_{k = 0}^n {( - 1)^{kq} \frac{n}{{n + k}}\binom{n + k}{n - k}F_{p + q}^{n - k} F_p^{n - k} F_q^{2k} \alpha^{k(2p + q) + t} }= \frac{{( - 1)^{nq} }}{2}F_{p + q}^{2n} \alpha^{qn + t}  + \frac{{( - 1)^t }}{2}F_p^{2n} \beta^{qn - t},
\end{equation*}
from which~\eqref{eq.sqpsz3a} and~\eqref{eq.uqixdew} follow. The proof of~\eqref{eq.tm4ir7g} and~\eqref{eq.h4xq8wt} is similar; in~\eqref{main_id4}, set
$\displaystyle
x=\sqrt{ \frac{5( - 1)^{q + 1}F_q^2 }{4L_p L_{p + q} }\alpha ^{2p + q}}
$.
\end{proof}
\begin{example}
If $n$ is a positive integer and $t$ is any integer, then
\begin{gather*}
\sum_{k = 0}^n ( - 2)^{n-k}\frac{n}{n + k}\binom{n + k}{n - k} L_{5k + t}   = \frac{1}{2}\big(L_{t-n}  + 4^{n} L_{n + t}\big),\\
\sum_{k = 0}^n (-2)^{n-k} \frac{n}{n + k}\binom{n + k}{n - k} F_{5k + t}   = \frac{1}{2}\big(F_{t-n}  + 4^{n} F_{n + t}\big),\\
\sum_{k = 0}^n ( -2)^{n-k} \frac{n}{n + k}\binom{n + k}{n - k} 5^k L_{k + t}  = \frac{1}{2}\big(4^{n} L_{t-n}+L_{n + t} \big),\\
\sum_{k = 0}^n ( - 2)^{n-k}\frac{n}{{n + k}}\binom{n + k}{n - k}5^kF_{k + t}   = \frac12 \big(4^{n} F_{t-n}+F_{n + t}\big).
\end{gather*}
\end{example}

\section{Binomial Fibonacci and Lucas sums from identities involving $U_n(x)$}

Using the fact that $\frac{d}{dx} T_n(x)=n U_{n-1}(x)$ we get from \eqref{main_id1}
\begin{equation}\label{main_id5}
\sum_{k=1}^n (-2)^{k}  \frac k{n + k}\binom{n + k}{n - k} (1 \mp x)^{k-1} = \mp U_{n-1} (x),
\end{equation}
and  from \eqref{main_id3} and \eqref{main_id4}
\begin{equation*}
\sum_{k=1}^n 4^k  \frac k{n + k}\binom{n + k}{n - k} (x^2-1)^{k-1} =  \frac{U_{2n-1} (x)}{x}
\end{equation*}
and
\begin{equation}\label{main_id7}
\sum_{k=1}^n (-4)^{k}\frac k{n + k}\binom{n + k}{n - k} x^{2k} = (-1)^n x U_{2n-1} (x).
\end{equation}

From here, setting $x=3/2$ in \eqref{main_id5} and using $U_{n-1}(3/2)=F_{2n}$ we get
\begin{equation*}
\sum_{k=1}^n  \frac k{n + k}\binom{n + k}{n - k} = \frac12 F_{2n}
\end{equation*}
and
\begin{equation*}
\sum_{k=1}^n (-5)^{k-1} \frac k{n + k}\binom{n + k}{n - k} = \frac{(-1)^{n-1}}{2} F_{2n}.
\end{equation*}
Also, with $x=\sqrt{5}/2$ upon combining we produce
\begin{equation*}
\sum_{k=1}^n (-1)^{k-1} \frac{k}{n + k}\binom{n + k}{n - k} L_{3(k-1)} = \frac{1-(-1)^n}{2}L_n,
\end{equation*}
where it was used that 
\begin{equation*}
U_n \Big (\frac{\sqrt{5}}{2} \Big ) = 
\begin{cases}
L_{n+1}, & \text{\rm $n$ even;} \\ 
\sqrt{5} F_{n+1}, & \text{\rm $n$ odd.} 
\end{cases}
\end{equation*}

In general, working with \eqref{main_id2} have the relations
\begin{equation*}
\sum_{k=1}^n (-2)^{k-1} \frac k{n + k}\binom{n + k}{n - k} \big(1 \mp T_m(x)\big)^{k-1} = (\pm 1)^{n-1} \frac{U_{nm-1}(x)}{2U_{m-1}(x)},
\end{equation*}
where we have used that $U_{m-1}\big(T_n(x)\big)=\frac{U_{mn-1}(x)}{U_{n-1}(x)}$.
\begin{lemma}\label{Lem.hspwn3b}
If $n$ is a positive integer  and $p$ an integer, then
\begin{align*}
U_n\Big(\frac{L_p}2\Big)&=\frac{F_{p(n + 1)}}{F_p},\quad\text{\rm $p$ even},\\
U_n\Big(\frac{iL_p}2\Big)&=\frac{i^nF_{p(n + 1)}}{F_p},\quad\text{\rm $p$ odd},\\
U_n \Big( {\frac{ \sqrt 5 F_p}{2}} \Big) &=  \begin{cases}
 L_{p(n + 1)} /L_p,&\text{\rm $p$ odd, $n$ even};  \\ 
 F_p F_{p(n + 1)} F_p /L_p,&\text{\rm $p$,  $n$ odd}; 
 \end{cases}\\
U_n \Big( {\frac{{i\sqrt 5 F_p  }}{2}}\Big) &=  \begin{cases}
 i^nL_{p(n + 1)} /L_p,&\text{\rm $p$, $n$ even};  \\ 
 i^n\sqrt 5 F_{p(n + 1)}  /L_p,&\text{\rm $p$ even, $n$ odd}.
 \end{cases}
\end{align*}
\end{lemma}
\begin{theorem}\label{thm.nl0qyn5}
If $n$ is a positive integer and $p$ is an integer, then
\begin{equation*}
\sum_{k = 1}^n (-1)^{(p-1)(n-k)} \frac {k}{n + k} \binom{n + k}{n - k} L_p^{2k - 1} = \frac{F_{2np}}{2F_p},\quad p\ne0,
\end{equation*}
\begin{equation*}
\sum_{k = 1}^n (-1)^{p(n-k)} \frac {k}{n + k} \binom{n + k}{n - k} 5^{k-1} F_p^{2k - 1} = \frac{F_{2np}}{2L_p}.
\end{equation*}
\end{theorem}
\begin{proof}
Evaluate \eqref{main_id7} at $x=L_p/2$, $x=iL_p/2$, $x=\sqrt 5F_p/2$ and $x=i\sqrt 5F_p/2$, in turn, using Lemma~\ref{Lem.hspwn3b}. This gives
\begin{gather*}
\sum_{k = 1}^n (\pm1)^{n-k}\frac k{n + k} \binom{n + k}{n - k} L_p^{2k - 1} = \frac{F_{2np}}{2F_p}\qquad\braces{\mbox{$p$ odd}}{\mbox{$p$ even}},\\
\sum_{k = 1}^n (\mp1)^{n-k} \frac k{n + k} \binom{n + k}{n - k} 5^{k - 1} F_p^{2k - 1} = \frac{F_{2pn}}{2L_p}\qquad\braces{\mbox{$p$ odd}}{\mbox{$p$ even}}.
\end{gather*}
This completes the proof.
\end{proof}
\begin{lemma}
If $n$ is a positive integer, then
\begin{gather*}
\sqrt {\frac{5}{{2\alpha }}}\, U_{2n - 1} \bigg( {\sqrt {\frac{{\alpha ^5 }}{8}} } \bigg) = 2^n\alpha^n  -(-1)^n \frac{{\beta^{n} }}{2^n},\nonumber\\
\sqrt {\frac{5}{{2\beta }}}\, U_{2n - 1} \bigg( {\sqrt {\frac{{\beta ^5 }}{8}} } \bigg) = 2^n\beta^n   - (-1)^n\frac{{\alpha ^{n} }}{{2^n}}\label{eq.qcelehc}.
\end{gather*}
\end{lemma}
\begin{theorem}
If $n$ is a positive integer and $t$ is any integer, then
\begin{gather*}
\sum_{k = 1}^n ( - 2)^{n-k} \frac{k}{{n + k}}\binom{n + k}{n - k}L_{5k + t}   =\frac12\left( 4^{n} F_{ t +n+ 3}  -  F_{t-n+3}\right)\!,\\
\sum_{k = 1}^n ( -2)^{n-k}\frac{k}{{n + k}}\binom{n + k}{n - k}F_{5k + t}   =  \frac1{10}\left(4^{n} L_{t+n+3} -  L_{n - t - 3}\right)\!.
\end{gather*}
\end{theorem}
\begin{proof}
Set $x=\sqrt{\beta^5/8}$ in~\eqref{main_id7} and use~\eqref{eq.qcelehc} to obtain
\begin{equation*}
\sqrt 5 \sum_{k = 1}^n {\frac{{( - 1)^{k - 1} }}{{2^k }}\frac{k}{{n + k}}\binom{n + k}{n - k}\beta ^{5k + t}}  = ( - 2)^{n - 1} \beta ^{n + t + 3}  + \frac{{( - 1)^t }}{{2^{n + 1} }}\alpha ^{n - t - 3}
\end{equation*}
from which the results follow.
\end{proof}
\begin{theorem}
If $n$ is a non-negative integer and $t$ is any integer, then
\begin{align*}
\sum_{k = 1}^n ( - 1)^{k - 1} {\frac{k}{{n + k}}\binom{n + k}{n - k} L_{3k + t} } & =  \begin{cases}
 \frac{1}{2}L_{t + 3} L_n,&\text{\rm $n$ odd};  \\ 
  - \frac{5}{2}F_{t + 3} F_n,&\text{\rm $n$ even};   
 \end{cases}\\
\sum_{k = 1}^n ( - 1)^{k - 1}{\frac{k}{{n + k}}\binom{n + k}{n - k} F_{3k + t} }  &=  \begin{cases}
 \frac{1}{2}F_{t + 3} L_n,&\text{\rm $n$ odd};  \\ 
  -\frac{1}{2} L_{t + 3} F_n,&\text{\rm $n$ even}. 
 \end{cases}
\end{align*}
\end{theorem}
\begin{proof}
Set $x=\sqrt{\alpha^3}/2$ in~\eqref{main_id7} and use 
\begin{equation*}
U_{2n - 1} \Big( {\frac{{\sqrt {\alpha ^3 } }}{2}} \Big) = \begin{cases}
\sqrt {\alpha ^3 } L_n,&\text{\rm  $n$ odd};  \\[4pt]
\sqrt {5\alpha ^3 } F_n, &\text{\rm $n$ even}, 
\end{cases}
\end{equation*}
to obtain
\begin{equation*}
\sum_{k = 1}^n ( - 1)^{k - 1} {\frac{k}{{n + k}}\binom{n + k}{n - k} \beta^{3k + t} }  = \frac{\beta^{t + 3}}{2}\cdot\begin{cases}
  L_n,&\text{$n$ odd};  \\ 
  \sqrt 5F_n,&\text{$n$ even}; 
 \end{cases}
\end{equation*}
from which the results follow.
\end{proof}
\begin{theorem}
If $n$ is a non-negative integer and $t$ is any integer, then
\begin{gather}
\sum_{k = 1}^n ( - 4)^{k - 1} {\frac{k}{{n + k}}\binom{n + k}{n - k} L_{k + t} }  =  \begin{cases}
 -\frac12 L_{t + 1} L_{3n},&\text{\rm $n$ odd};  \\ 
  \frac52F_{t + 1} F_{3n},&\text{\rm $n$ even};  \\ 
 \end{cases}\\
\sum_{k = 1}^n ( - 4)^{k - 1}{\frac{k}{{n + k}}\binom{n + k}{n - k}F_{k + t} }  =  \begin{cases}
 - \frac12 F_{t + 1} L_{3n},&\text{\rm $n$ odd};  \\ 
  \frac12 L_{t + 1} F_{3n},&\text{\rm $n$  even}. 
 \end{cases}
\end{gather}
\end{theorem}
\begin{proof}
Set $x=\sqrt{\alpha}$ in~\eqref{main_id7} and use 
\begin{equation*}
U_{2n - 1}\big( \sqrt\alpha\big) =\begin{cases}
\frac{\sqrt\alpha }2 L_{3n},&\text{\rm $n$ odd};  \\[2pt] 
\frac{\sqrt{5\alpha} }2 F_{3n}, &\text{\rm $n$  even}, 
\end{cases}
\end{equation*}
to obtain
\begin{equation*}
\sum_{k = 1}^n ( - 4)^{k - 1}{\frac{k}{{n + k}}\binom{n + k}{n - k}\beta^{k + t} }  = -\frac{\beta^{t + 1}}2\cdot\begin{cases}
  L_{3n},&\text{$n$ odd};  \\ 
\sqrt 5  F_{3n},&\text{$n$ even}; 
 \end{cases}
\end{equation*}
from which the results follow.
\end{proof}
%
\begin{theorem}
If $n$ is a non-negative integer and $p$ is a non-zero integer, then
\begin{equation*}\label{eq.v6rb7fk}
\sum_{k = 1}^n {( - 1)^{n-k} \frac{k}{{n + k}}\binom{n + k}{n - k}\left( {\frac{{2L_{2p} }}{\sqrt5 F_{2p}}} \right)} ^{2k}  = ( - 1)^{p} \frac{L_{2p}(L_p^{4n}  - 5^{2n} F_p^{4n} )}{4 (\sqrt5 F_{2p})^{2n}}.
\end{equation*}
\end{theorem}
\begin{proof}
Set $x=L_{2p}/(\sqrt 5F_{2p})$ in \eqref{main_id7} and use 
\begin{equation}\label{kc1e6ox}
U_{n - 1} \Big( \frac{{L_{2p} }}{\sqrt 5F_{2p}  } \Big) = \frac{{( - 1)^p(L_p^{2n}  - 5^n F_p^{2n} )}}{{4(\sqrt 5F_{2p}  )^{n - 1} }}.
\end{equation}
\end{proof}
\begin{theorem}
If $n$ is a non-negative integer and $p$ and $q$ are non-zero integers, then
\begin{equation*}
\sum_{k = 1}^n ( - 1)^{(p - q)(n-k)} {\frac{k}{{n + k}}\binom{n + k}{n - k} \left( {\frac{{F_{p + q} F_{p - q} }}{{F_p F_q }}} \right)^{2k} } = \frac{{(F_p^{4n}  - F_q^{4n} )F_{p + q} F_{p - q} }}{2\big(F_p^2  + ( - 1)^{p - q} F_q^2 \big)F_q^{2n} F_p^{2n} }.
\end{equation*}
\end{theorem}
\begin{proof}
Set $x=\displaystyle\frac{(-i)^{p - q + 1}F_{p + q} F_{p - q}}{2 F_q F_p }$ in~\eqref{main_id7} and use
\begin{equation}\label{eq.ormac9f}
U_{n - 1} \left( \frac{i^{p - q + 1} F_{p + q} F_{p - q}}{2F_qF_p} \right) = \frac{F_p^{2n}  + ( - 1)^{p - q} F_q^{2n} }{F_q^{n - 1} F_p^{n - 1}\big(F_p^2  + ( - 1)^{p - q} F_q^2\big) }.
\end{equation}
\end{proof}
\begin{lemma} We have 
\begin{equation}
\begin{split}
 &\hspace{-1.5cm}(-1)^{(q+1)(n+3/2)+\lfloor q/2\rfloor+1} \alpha ^{p + q/2} \,U_{2n - 1}\!\left( {\sqrt{\frac{( - 1)^{q + 1} F_q^2}{4F_p F_{p + q}}\alpha ^{2p + q} }} \right)\label{eq.r6f98dl}\\
& = \frac{{F_{p + q}^{2n + 1} \alpha ^{qn + p}  - F_p^{2n} F_{p + q} }(-\beta)^{qn - p}+F_{p + q}^{2n} F_p \alpha ^{qn + p + q}  - F_p^{2n + 1} (-\beta) ^{qn - p - q} }{\big(F_p F_{p + q}\big)^{n-1/2}\big(F_{p + q}^2  + F_p F_{p + q} L_q  + ( - 1)^qF_p^2\big) },
\end{split}
\end{equation}
\begin{equation*}
\begin{split}
&\hspace{-1.5cm}(-1)^{(q+1)(n+5/2)+\lfloor q/2\rfloor+p} \beta ^{p + q/2}\, U_{2n - 1}\! \left( \sqrt{\frac{( - 1)^{q + 1}F_q^2 }{{4F_p F_{p + q} }}\beta ^{2p + q}} \right)\\
& = \frac{F_{p + q}^{2n + 1} \beta ^{qn + p}  - F_p^{2n} F_{p + q} (-\alpha)^{qn - p} + F_{p + q}^{2n} F_p \beta ^{qn + p + q}  - F_p^{2n + 1} (-\alpha) ^{qn - p - q} }{\big(F_pF_{p + q}\big)^{n-1/2}\big(F_{p + q}^2  + F_p F_{p + q} L_q  + ( - 1)^qF_p^2\big) },
\end{split}
\end{equation*}
\begin{equation*}
\begin{split}
&\hspace{-1.5cm}(-1)^{(q+1)(n+3/2)+\lfloor q/2\rfloor+1} \alpha ^{p + q/2}\, U_{2n - 1}\! \left( \sqrt{\frac{( - 1)^{q + 1}5 F_q^2 }{{4L_p L_{p + q} }}\alpha ^{2p + q}} \right)\\
& = \frac{L_{p + q}^{2n + 1} \alpha ^{qn + p}  - L_p^{2n} L_{p + q} (-\beta)^{qn - p} + L_{p + q}^{2n}L_p \alpha ^{qn + p + q}  - L_p^{2n + 1} (-\beta) ^{qn - p - q} }{\big(L_pL_{p + q}\big)^{n-1/2} \big(L_{p + q}^2  + L_p L_{p + q} L_q  + ( - 1)^qL_p^2\big) },\end{split}
\end{equation*}
\begin{equation*}
\begin{split}
&\hspace{-1.5cm}(-1)^{(q+1)(n+3/2)+2\lfloor q/2\rfloor+p}  \beta ^{p + q/2} \,U_{2n - 1} \!\left( \sqrt{\frac{{( - 1)^{q + 1}5F_q^2 }}{4L_p L_{p + q} }\beta ^{2p + q}}\right)\\
& = \frac{L_p^{2n+1}\beta^{qn+p}- L_{p}^2n L_{p+q}(-\alpha)^{qn - p}+ L_{p+q}^{2n} L_{p}\beta ^{qn + p+q} -L_p^{2n + 1} (-\alpha) ^{qn - p - q}}{(L_pL_{p + q})^{n-1/2}\big(L_{p + q}^2  + L_p L_{p + q} L_q  + ( - 1)^qL_p^2\big)}.
\end{split}
\end{equation*}
\end{lemma}
\begin{theorem}\label{thm.kyu2oro}
If $n$ is a positive integer and $p$, $q$ and $t$ are any integers, then
\begin{equation}
\begin{split}
\sum_{k = 1}^n & ( - 1)^{q(n-k)} \frac{k}{{n + k}}\binom{n + k}{n - k}F_p^{n - k} F_{p + q}^{n - k} F_q^{2k-1} L_{(2p + q)k + t} \\ 
& = \frac{ F_{p + q}^{2n}\big(F_{p + q} L_{ p + t + qn} + F_p L_{ q + p + t + qn}\big)         - F_p^{2n}\big( F_{p + q} L_{p + t - qn}   + F_p L_{ q + p + t -  qn}\big)}{2\big({F_{p + q}^2  + F_p F_{p + q} L_q  + ( - 1)^q F_p^2 }\big)}, 
\label{eq.tpynnr7}
\end{split}
\end{equation}
\begin{equation}
\begin{split}
\sum_{k = 1}^n & ( - 1)^{q(n-k)} \frac{k}{n + k}\binom{n + k}{n - k}F_p^{n - k} F_{p + q}^{n - k} F_q^{2k-1} F_{(2p + q)k + t}\\ 
& = \frac{F_{p +q}^{2n}\big(F_{p + q} F_{ p + t +qn}   + F_p F_{q + p + t+qn }\big) - F_p^{2n}\big( F_{p + q} F_{p+t -qn}  + F_pF_{p + q + t - qn}\big) }{2\big(F_{p + q}^2  + F_p F_{p + q} L_q  + ( - 1)^q F_p^2 \big)}, \label{eq.gyr3ptz}
\end{split}
\end{equation}
\begin{equation}
\begin{split}
\sum_{k = 1}^n &( - 1)^{q(n-k)} \frac{k}{{n + k}}\binom{n + k}{n - k}L_p^{n - k} L_{p + q}^{n - k} F_q^{2k-1} 5^kF_{(2p + q)k + t} \\ 
& = \frac{ L_{p + q}^{2n}\big(L_{p + q} L_{ p + t + qn}   + L_p L_{q + p + t+qn}\big) -  L_p^{2n}\big( L_{p + q} L_{p+t -qn} + L_p L_{p + q + t - qn}\big) }{2\big(L_{p + q}^2  + L_p L_{p + q} L_q  + ( - 1)^q L_p^2 \big)},
\end{split}\label{eq.z2nmrct}
\end{equation}
\begin{equation}
\begin{split}
\sum_{k = 1}^n & ( - 1)^{q(n-k)} \frac{k}{{n + k}}\binom{n + k}{n - k}L_p^{n - k} L_{p + q}^{n - k} F_q^{2k-1} 5^{k - 1}L_{(2p + q)k + t} \\ 
& = \frac{L_{p + q}^{2n}\big(L_{p + q} F_{qn + p + t}  +  L_p F_{ q + p + t + qn}\big)  - L_p^{2n}\big( L_{p + q} F_{p+t-qn}  + L_p F_{p + q + t - qn }\big)}{{2\big(L_{p + q}^2  + L_p L_{p + q} L_q  + ( - 1)^q L_p^2 \big)}}.\label{eq.rankav6}
\end{split}
\end{equation}
\end{theorem}
\begin{proof} In~\eqref{main_id7} set
$\displaystyle 
x=\sqrt{ \frac{( - 1)^{q + 1}F_q^2 }{{4F_p F_{p + q} }}\alpha ^{2p + q}}
$
and use~\eqref{eq.r6f98dl} to obtain
\begin{equation*}
\begin{split}
&2\sum_{k = 1}^n {( - 1)^{qk} \frac{k}{{n + k}}\binom{n + k}{n - k}F_p^{n - k} F_{p + q}^{n - k} F_q^{2k}\, \alpha^{(2p + q)k + t} } = \frac{{F_q }}{2\big(F_{p + q}^2  + F_p F_{p + q} L_q  + ( - 1)^q F_p^2\big)} \\
&\qquad\qquad\quad\times\!\big( {( - 1)^{nq} F_{p + q}^{2n + 1}\, 2\alpha^{qn + p + t}  - ( - 1)^{p + t} F_p^{2n} F_{p + q}\, 2\beta^{qn - p - t} } \\ 
&\qquad\qquad\qquad\quad  + ( - 1)^{nq} F_{p + q}^{2n} F_p\, 2\alpha^{qn + q + p + t}  - ( - 1)^{p + q + t} F_p^{2n + 1}\, 2\beta^{qn - p - q - t} \bigr),
\end{split}
\end{equation*}
from which~\eqref{eq.tpynnr7} and~\eqref{eq.gyr3ptz} follow. The proof of~\eqref{eq.z2nmrct} and~\eqref{eq.rankav6} is similar; in~\eqref{main_id7}, set
$\displaystyle
x\!\!~=~\!\!\sqrt{ \frac{{( - 1)^{q + 1}5F_q^2 }}{{4L_p L_{p + q} }}\alpha ^{2p + q} }.
$
\end{proof}

\section{Some binomial-coefficient weighted Fibonacci and \break Lucas sums}

In this section we present some Fibonacci and Lucas sums having binomial coefficients as weights. 
\begin{theorem}
If $n$ is a positive integer and $p$ is an integer, then
\begin{align*}
&\sum_{k = 0}^n (- 1)^{(p-1)(n-k)}\binom{n + k}{n - k}  L_p^{2k} = \frac{F_{(2n + 1)p}}{F_p}, \\
&\sum_{k = 0}^n ( - 1)^{p(n-k)} \binom{n + k}{n - k}  5^k F_p^{2k} = \frac{L_{(2n + 1)p}}{L_p}.
\end{align*}
\end{theorem}
\begin{proof}
Addition of each identity stated in Theorem~\ref{thm_xxx} and its counterpart in Theorem~\eqref{thm.nl0qyn5} while making use of
$L_rF_s + L_sF_r = 2F_{r + s}$ and $L_rL_s + 5F_rF_s = 2L_{r + s}$.
\end{proof}

The identities stated below in Theorem~\ref{thm.kihhjta} follow immediately upon addition of each identity in Theorem~\ref{thm.awwyoac} to the respective corresponding identity in Theorem~\ref{thm.kyu2oro}.
\begin{theorem}\label{thm.kihhjta}
If $n$ is a non-negative integer and $p$, $q$ and $t$ are any integers, then
\begin{align*} 
\sum_{k = 0}^n &( - 1)^{q(n-k)} \binom{n + k}{n - k}F_p^{n - k} F_{p + q}^{n - k} F_q^{2k} L_{(2p + q)k + t} = \frac{1}{2}\left(F_{p + q}^{2n} L_{t+qn}  + F_p^{2n} L_{t-qn}\right)\\
& +\frac{F_q\big( F_{p + q}^{2n}( F_{p + q}L_{ p + t+qn }  + F_p L_{q + p + t +qn}) - F_p^{2n}(F_{p + q} L_{p+t- qn}   +  F_p L_{q + p + t -qn}) \big)}{{2\big(F_{p + q}^2  + F_p F_{p + q} L_q  + ( - 1)^q F_p^2\big) }},\\
\sum_{k = 0}^n &( - 1)^{q(n-k)} \binom{n + k}{n - k}F_p^{n - k} F_{p + q}^{n - k} F_q^{2k} F_{(2p + q)k + t} = \frac{1}{2}\left(F_{p + q}^{2n} F_{t+qn}  - F_p^{2n} F_{t-qn}\right)\\
&+\frac{F_q\big( F_{p + q}^{2n}(F_{p + q} F_{ p + t+qn}+ F_p F_{ q + p + t +qn})   - F_p^{2n}(F_{p + q} F_{t+p - qn}   + F_p F_{q + p + t  - qn}) \big)}{2\big(F_{p + q}^2  + F_p F_{p + q} L_q  + ( - 1)^q F_p^2\big)},\label{eq.htvwhw2}\\
\sum_{k = 0}^n & ( - 1)^{q(n-k)} \binom{n + k}{n - k}L_p^{n - k} L_{p + q}^{n - k} F_q^{2k} 5^kF_{(2p + q)k + t}= \frac{1}{2}\left(L_{p + q}^{2n} F_{t+qn}  + L_p^{2n}F_{t-qn}\right)\\
&+\frac{F_q\big( L_{p + q}^{2n}(L_{p + q} L_{ p + t+qn} + L_p L_{q + p + t+qn}) - L_p^{2n}(L_{p + q} L_{p + t- qn}   + L_pL_{q + p + t- qn}) \big)} {2\big(L_{p + q}^2  + L_p L_{p + q} L_q  + ( - 1)^q L_p^2\big) },\\
\sum_{k = 0}^n &( - 1)^{q(n-k)} \binom{n + k}{n - k}L_p^{n - k} L_{p + q}^{n - k} F_q^{2k} 5^kL_{(2p + q)k + t} = \frac{1}{2}\left(L_{p + q}^{2n} L_{t+qn}  + L_p^{2n} L_{t-qn}\right)\\
&+\frac{5F_q\left(L_{p + q}^{2n + 1}(L_{p + q} F_{ p + t+qn } + L_p F_{ q + p + t+qn})  - L_p^{2n}(L_{p + q} F_{ p + t-qn }  + L_p F_{q + p + t- qn})  \right)}{2(L_{p + q}^2  + L_p L_{p + q} L_q  + ( - 1)^q L_p^2)}.
\end{align*}
\end{theorem}
\begin{example}
If $n$ is a non-negative integer and $t$ is any integer, then
\begin{gather*}
\sum_{k = 0}^n {( - 1)^{n-k} \binom{n + k}{n - k}L_{3k + t} }  = \begin{cases}
 L_{t + n}+  L_{t + 1} L_n,&\text{\rm $n$ odd};  \\ 
  L_{t + n}  + 5F_{t + 1} F_n,&\text{\rm  $n$ even};  \\ 
 \end{cases}\\
\sum_{k = 0}^n {( - 1)^{n-k} \binom{n + k}{n - k}F_{3k + t} }  = \begin{cases}
 F_{t + n} + F_{t + 1} L_n,&\text{\rm  $n$ odd};  \\ 
  F_{t + n}  + L_{t + 1} F_n,&\text{\rm  $n$ even}. 
 \end{cases}
\end{gather*}
In particular,
\begin{gather*}
\sum_{k = 0}^n {( - 1)^{n-k} \binom{n + k}{n - k}L_{3k} }  = \begin{cases}
 2L_n ,&\text{\rm $n$ odd};  \\ 
  2L_{n + 1},&\text{\rm $n$ even};  \\ 
 \end{cases}\\
\sum_{k = 0}^n {( - 1)^{n-k} \binom{n + k}{n - k}F_{3k} }  = \begin{cases}
 2F_{n + 1},&\text{\rm $n$  odd};  \\ 
  2F_n,&\text{\rm $n$ even};  
 \end{cases}\\
\sum_{k = 0}^n {( - 1)^{n-k} \binom{n + k}{n - k}L_{3k - 1} }  = \begin{cases}
 L_{n + 2} ,&\text{\rm $n$ odd};  \\ 
  L_{n - 1},&\text{\rm $n$ even};  \\ 
 \end{cases}\\
\sum_{k = 0}^n {( - 1)^{n-k} \binom{n + k}{n - k}F_{3k - 1} }  = \begin{cases}
 F_{n - 1},&\text{\rm$n$ odd};  \\ 
  F_{n + 2},&\text{\rm$n$ even}.  
 \end{cases}
\end{gather*}
\end{example}
\begin{example}
If $n$ is a non-negative integer and $t$ is any integer, then
\begin{gather*}
\sum_{k = 0}^n {( -2)^{n-k} \binom{n + k}{n - k} L_{5k + t} }  = 2^{2n + 1} F_{ t +n +  1} - F_{t-n},\\
\sum_{k = 0}^n {( - 2)^{n-k} \binom{n + k}{n - k} F_{5k + t} }  =\frac15\left( 2^{2n + 1} L_{ t +n+ 1}- L_{t- n}\right)\!.
\end{gather*}
\end{example}

The identities in this section and many similar results can be obtained directly from Lemma~\ref{lem.x6c6bl9} which is a consequence of~\eqref{eq.w83lskz}, \eqref{main_id4} and~\eqref{main_id7}.
\begin{lemma}\label{lem.x6c6bl9}
If $n$ is a non-negative integer and $x$ is a complex variable, then
\begin{equation}\label{eq.f8j08d9}
\sum_{k=0}^n{(-4)^{k}\binom{n + k}{n - k}x^{2k}}=(-1)^nU_{2n}(x).
\end{equation}
\end{lemma}
\begin{theorem}
If $n$ is a non-negative integer and $p$ is a non-zero integer, then
\begin{equation*}
\sum_{k = 0}^n {( - 1)^{k} \binom{n + k}{n - k}\left( {\frac{2L_{2p}}{\sqrt5F_{p}}} \right)} ^{2k} = ( - 1)^{n-p}\,\frac{L_p^{4n + 2}  - 5^{2n + 1} F_p^{4n + 2}}{4\cdot5^n F_{2p}^{2n}  }.
\end{equation*}
\end{theorem}
\begin{proof}
Set $ x=L_{2p}/(\sqrt 5 F_{2p})$ in~\eqref{eq.f8j08d9} and use \eqref{kc1e6ox}.
\end{proof}
\begin{theorem}
If $n$ is a non-negative integer and $p$ and $q$ are non-zero integers, then
\begin{equation*}
\sum_{k = 0}^n {( - 1)^{(p - q)(n-k)} \binom{n + k}{n - k} \left( {\frac{{F_{p + q} F_{p - q} }}{{F_p F_q }}} \right)^{2k} }  = \frac{{F_p^{4n + 2}  + ( - 1)^{p - q} F_q^{4n + 2}}}{F_q^{2n} F_p^{2n}(F_p^2  + ( - 1)^{p - q} F_q^2 )}.
\end{equation*}
\end{theorem}
\begin{proof}
Set  $\displaystyle x=\frac{(-i)^{p - q + 1}F_{p + q} F_{p - q}}{2 F_q F_p}$ in \eqref{eq.f8j08d9} and use \eqref{eq.ormac9f}.
\end{proof}

\section{Related combinatorial sums}

This section does not deal with Fibonacci numbers but is inspired by a recent article by Chu and Guo \cite{ChuGuo}. In this article the authors study combinatorial sums of the form
\begin{equation*}
\sum_{k=0}^n (-1)^k \frac{\binom {n+\lambda}{2k+\delta}}{\binom {n}{k}},
\end{equation*}
where $\lambda,n\in\mathbb{N}_0$ and $\delta\in\{0,1\}$. We show how these sums are related to the sums studied in this paper via Theorem \ref{thm1}.
\begin{theorem}\label{comb_thm}
For $0<n+m\leq 2n+1$ we have the relation
\begin{equation*}
\sum_{k=0}^{n+m}  \frac{(-2)^k}{(n+m+k)(n-m+k+2)}\frac{\binom {n+m+k}{n+m-k}}{\binom {n-m+k+1}{k}}
= \frac{1}{2(n+1)(n+m) }\sum_{k=0}^n (-1)^k \frac{\binom {n+m}{2k}}{\binom {n}{k}}.
\end{equation*}
In particular,
\begin{gather*}
\sum_{k=0}^{n} \frac{(-2)^k}{(n+k)(n+k+1)(n+k+2)}\frac{\binom {n+k}{n-k}}{\binom {n+k}{k}}
= \frac{1}{2n(n+1)^2 }\sum_{k=0}^n (-1)^k \frac{\binom {n}{2k}}{\binom {n}{k}},\\
\sum_{k=0}^{2n}  \frac{(-2)^k}{(2n+k)(k+1)(k+2)} \binom {2n+k}{2n-k}
=\frac{1}{4n(n+1)} \sum_{k=0}^n (-1)^k \frac{\binom {2n}{2k}}{\binom {n}{k}}.
\end{gather*}
\end{theorem}
\begin{proof}
From Theorem \ref{thm1} upon replacing $n$ by $n+m$ we get
\begin{equation*}
\sum_{k=0}^{n+m} (-2)^k \frac{n+m}{n + m + k} \binom{n + m + k}{n + m - k} (1-x)^k = T_{n+m}(x).
\end{equation*}
Multiplying through by $x^{n-m+1}$ and integrating from 0 to 1 results in
\begin{equation*}
\sum_{k=0}^{n+m} (-2)^k \frac{n+m}{n + m + k} \binom{n + m + k}{n + m - k} \int_0^1 x^{n-m+1} (1-x)^k dx = \int_0^1 x^{n-m+1} T_{n+m}(x)dx.
\end{equation*}

The left-hand side evaluated using the beta function $B(a,b)$ given by
\begin{equation*}
B(a,b) = \int_0^1 x^{a-1} (1-x)^{b-1} dx,
\end{equation*}
and equals
\begin{equation*}
\sum_{k=0}^{n+m} (-2)^k\frac{n+m}{n + m + k} \binom{n + m + k}{n + m - k} B(n-m+2,k+1)
\end{equation*}
with
\begin{equation*}
B(n-m+2,k+1) = \frac{1}{n-m+k+2}\binom {n-m+k+1}{k}^{-1}.
\end{equation*}

The right-hand side is evaluated similarly using the representation from \eqref{ChebRep1}. We have
\begin{align*}
\int_0^1 x^{n-m+1} T_{n+m}(x)dx & =  \sum_{k=0}^{\lfloor{(n+m)/2}\rfloor} (-1)^k \binom {n+m}{2k} \int_0^1 x^{2n-2k+1} (1-x^2)^k dx \\
& = \frac{1}{2} \sum_{k=0}^{\lfloor{(n+m)/2}\rfloor} (-1)^k \binom {n+m}{2k} \int_0^1 z^{n-k} (1-z)^k dz \quad (z=x^2) \\
& = \frac{1}{2} \sum_{k=0}^{\lfloor{(n+m)/2}\rfloor} (-1)^k \binom {n+m}{2k} B(n-k+1,k+1) \\
& = \frac{1}{2(n+1)} \sum_{k=0}^{\lfloor{(n+m)/2}\rfloor} (-1)^k \frac{\binom {n+m}{2k}}{\binom {n}{k}}.
\end{align*}
The final expression follows as $\binom {n}{k}=0$ if $k>n$.
\end{proof}

Keeping in mind that
$$ 
\frac{n}{n + k} \binom{n + k}{n - k} = \frac{\binom {n+k-1}{k} \binom {n}{k}}{\binom {2k}{k}},
$$
Theorem \ref{comb_thm} is actually an identity containing four binomial coefficients
\begin{equation*}
 \sum_{k=0}^{n+m}  \frac{(-2)^k}{n-m+k+2}\frac{\binom {n+m+k-1}{k}}{\binom {n-m+k+1}{k}}\frac{\binom {n+m}{k}}{\binom {2k}{k}}
= \frac{1}{2(n+1)}\sum_{k=0}^n (-1)^k \frac{\binom {n+m}{2k}}{\binom {n}{k}}.
\end{equation*}

The special case $m=0$ gives the combinatorial relation valid for all $n\geq 1$
\begin{equation*}
\sum_{k=0}^{n}  \frac{(-2)^k}{(n+k)(n+k+1)(n+k+2)} \frac{\binom {n}{k}}{\binom {2k}{k}}
= \frac{1}{2n(n+1)} - \frac{1}{4(n+1)^2} \sum_{k=0}^n \frac{1}{\binom {n}{k}},
\end{equation*}
as Chu and Guo have shown \cite[Proposition 5]{ChuGuo} that
\begin{equation*}
\sum_{k=0}^n (-1)^k \frac{\binom {n}{2k}}{\binom {n}{k}} = n + 1 - \frac{n}{2} \sum_{k=0}^n \frac{1}{\binom {n}{k}}.
\end{equation*}

We continue with some more examples based on Chu and Guo's results. 
\begin{example}
When $m=1,2,3, 4$, then we can use the following evaluation from \cite{ChuGuo}
\begin{gather*}
\sum_{k=0}^n (-1)^k \frac{\binom {n+1}{2k}}{\binom {n}{k}} = 2 - \sum_{k=0}^n \frac{1}{\binom {n}{k}},\qquad 
\sum_{k=0}^n (-1)^{k+1} \frac{\binom {n+2}{2k}}{\binom {n}{k}} = \frac{2}{n},\\
\sum_{k=0}^n (-1)^{k+1} \frac{\binom {n+3}{2k}}{\binom {n}{k}} = \frac{2(n-3)}{n(n-1)},\qquad
\sum_{k=0}^n (-1)^{k+1} \frac{\binom {n+4}{2k}}{\binom {n}{k}} =  \frac{2(n^2-7n+16)}{(n-2)(n-1)n},
\end{gather*}
to get the following summation formulas: 
\begin{gather*}
\sum_{k=0}^{n+1}  \frac{(-2)^k }{(n+1+k)^2}\frac{\binom {n+1+k}{n+1-k}}{\binom {n+k}{k}}
= \frac{1}{2(n+1)^2}\Big ( 2 - \sum_{k=0}^n \frac{1}{\binom {n}{k}} \Big ),\quad n\geq 0,\\
\sum_{k=0}^{n+2}  \frac{(-2)^{k} }{(n+k)(n+2+k)}\frac{\binom {n+2+k}{n+2-k}}{\binom {n-1+k}{k}}
= \frac{-1}{n(n+1)(n+2)},\quad n\geq 1,\\
\sum_{k=0}^{n+3}  \frac{(-2)^{k} }{(n-1+k)(n+3+k)}\frac{\binom {n+3+k}{n+3-k}}{\binom {n-2+k}{k}}
= \frac{3-n}{(n-1)n(n+1)(n+3)},\quad n\geq 2,\\
\sum_{k=0}^{n+4}  \frac{(-2)^{k} }{(n-2+k)(n+4+k)}\frac{\binom {n+4+k}{n+4-k}}{\binom {n-3+k}{k}}
= \frac{-n^2+7n-16}{(n-2)(n-1)n(n+1)(n+4)}, \quad n\geq 3.
\end{gather*}
\end{example}

Working with the second part of Theorem \ref{thm1}, applying the same arguments but using the elementary integral
\begin{equation*}
\int_0^1 x^{n-m+1} (1+x)^k dx = \sum_{j=0}^k  \frac{\binom {k}{j}}{n-m+2+j}
\end{equation*}
yields the next theorem involving a double sum.
\begin{theorem}\label{comb_thm2}
For $m-1\leq n$, we have the relation
\begin{align*}
 \sum_{j=0}^{n+m}\sum_{k=0}^{n+m-j} & \frac{(-2)^{k+j}}{(n+m+k+j)(n-m+j+2)} \binom {n+m+k+j}{n+m-k-j} \binom {k+j}{j}
 \nonumber \\
& \qquad = \frac{(-1)^{n+m}}{2(n+1)(n+m) } \sum_{k=0}^n (-1)^k \frac{\binom {n+m}{2k}}{\binom {n}{k}}.
\end{align*}
In particular,
\begin{gather*}
\sum_{j=0}^{n}\sum_{k=0}^{n-j}  \frac{(-2)^{k+j}}{(n+k+j)(n+j+2)} \binom {n+k+j}{n-k-j} \binom {k+j}{j} 
= \frac{(-1)^{n}}{2n(n+1)} \sum_{k=0}^n (-1)^k \frac{\binom {n}{2k}}{\binom {n}{k}},\\
\sum_{j=0}^{2n}\sum_{k=0}^{2n-j}  \frac{(-2)^{k+j}}{(2n+k+j)(j+2)} \binom {2n+k+j}{2n-k-j} \binom {k+j}{j} 
= \frac{1}{4n(n+1)}\sum_{k=0}^n (-1)^k \frac{\binom {2n}{2k}}{\binom {n}{k}}.
\end{gather*}
\end{theorem}

Similar double sums involving ratios of binomial coefficients were studied recently by
Stenlund and Wan \cite{Sten}.

The analogous identities involving $\binom{n+m}{2k+1}$ should follow from \eqref{main_id5}
in conjunction with \eqref{ChebRep2}. As we do not want to overextend the length of the paper
we leave the details for a personal study.


\begin{thebibliography}{99}


\bibitem{Abd}
W.\,M. Abd-Elhameed, H.\,M. Ahmed, A. Napoli and V.~Kowalenko, New formulas involving Fibonacci and certain orthogonal polynomials, {\it Symmetry} \textbf{15} (2023), 736. 

\bibitem{Adegoke1}
K.~Adegoke, Weighted sums of some second-order sequences, {\it Fibonacci Quart.} \textbf{56}(3) (2018), 252--262.

\bibitem{Adegoke2}
K.~Adegoke, A.~Olatinwo and S.~Ghosh, Cubic binomial Fibonacci sums, {\it Electron. J. Math.} \textbf{2} (2021), 44--51.

\bibitem{Adegoke3}
K.~Adegoke, R.~Frontczak and T.~Goy, New binomial Fibonacci sums, preprint, 2022. https://arxiv.org/abs/2210.12159v1

\bibitem{Bai}
M.~Bai, W.~Chu and D.~Guo, Reciprocal formulae among Pell and Lucas polynomials, {\it Mathematics} \textbf{10} (2022), 2691.

\bibitem{ChuGuo}
W.~Chu and D.~Guo, Alternating sums of binomial quotients, {\it Math. Commun.} \textbf{27} (2022), 203--213.

\bibitem{Fan}
Z.~Fan and W.~Chu, Convolutions involving Chebyshev polynomials, {\it Electron. J. Math.} \textbf{3}, (2022), 38–46.

\bibitem{Frontczak}
R.~Frontczak and T.~Goy, Chebyshev-Fibonacci polynomial relations using generating functions, {\it Integers} \textbf{21} (2021), \#A100.  


\bibitem{Gao}
J.~Gao, Some new identities for arctangents
and Chebyshev polynomials, {\it J. Integer Seq.} \textbf{26} (2023), Article 23.1.3.


\bibitem{Gould}
H.\,W.~Gould, The Girard--Waring power sums formulas for symmetric functions and Fibonacci sequences, {\it Fibonacci Quart.} \textbf{37}(2) (1999), 135--139.

\bibitem{Jennings}
D.~Jennings, Some polynomial identities for the Fibonacci and Lucas numbers, {\it Fibonacci Quart.} \textbf{31}(2) (1993), 134--137.

\bibitem{Kilic1}
E.~Kilic and E.\,J.~Ionascu, Certain binomial sums with recursive coefficients, {\it Fibonacci Quart.} \textbf{48} (2) (2010), 161--167.

\bibitem{Kilic2}
E.~K\i l\i c, S.~Koparal and N.~\"{O}m\"{u}r, Powers sums of the first and second kinds of Chebyshev polynomials, {\it Iran. J. Sci. Technol. Trans. Sci.} \textbf{44} (2020), 425--435.

\bibitem{Kim}
T.~Kim, D.\,S. Kim, D.\,V. Dolgy and J.-W. Park, Sums of finite products of Chebyshev polynomials of the second kind and of Fibonacci polynomials, {\it J. Inequal. Appl.} \textbf{2018} (2018), Article ID 148.

\bibitem{Koshy} 
T. Koshy, Fibonacci and Lucas Numbers with Applications, Wiley-Interscience, 2001.

\bibitem{Li-Wenpeng}
C. Li and Z. Wenpeng, Chebyshev polynomials and their some interesting applications, {\it Adv. Differ. Equ.} \textbf{2017} (2017), 
Article 303. 
	
\bibitem{Li}
Y. Li, On Chebyshev polynomials, Fibonacci polynomials, and their derivatives, {\it J. Appl. Math.} \textbf{2014} (2014), Article ID 451953. 
	
\bibitem{Mason}	
J.\,C. Mason and D.\,C. Handscomb, Chebyshev Polynomials, CRC Press, Boca Raton, 2002.

\bibitem{OEIS}
N.\,J.\,A. Sloane (ed.), \textit{The On-Line Encyclopedia of Integer Sequences}, https://oeis.org.



\bibitem{Sten}
D. Stenlund and J.\,G. Wan, Some double sums involving ratios of binomial coefficients arising from urn models, {\it J. Integer Seq.} \textbf{22} (2019), Article 19.1.8. 

\bibitem{Vajda} 
S. Vajda, Fibonacci and Lucas Numbers, and the Golden Section: Theory and Applications, Dover Press, 2008.

\bibitem{Zhang}
W.\,P. Zhang, On Chebyshev polynomials and Fibonacci numbers, {\it Fibonacci Quart.} \textbf{40}(5) (2002), 424--428.



\end{thebibliography}
\end{document}